\documentclass[a4paper, 12pt, reqno]{amsart}

\usepackage{amssymb}
\usepackage[colorlinks=true]{hyperref}
\usepackage{thm-restate}

\usepackage{cleveref}
\usepackage{paralist}

\theoremstyle{plain}
  \newtheorem{thm}{Theorem}[section]
  \newtheorem*{thm*}{Theorem}

  \newtheorem{cor}[thm]{Corollary}
\theoremstyle{definition}
  \newtheorem{dfn}[thm]{Definition}
  \newtheorem{exmp}[thm]{Example}
\theoremstyle{remark}
  \newtheorem{rem}[thm]{Remark}

\let\opn\operatorname 

\numberwithin{equation}{section}
\newcommand{\NN}{\mathbb{N}} 
\newcommand{\ZZ}{\mathbb{Z}} 
\newcommand{\RR}{\mathbb{R}} 
 
\newcommand{\kk}{\Bbbk}		
\newcommand{\ba}{\mathbf{a}}

\newcommand{\mm}{\mathfrak{m}} 
\newcommand{\nn}{\mathfrak{n}}
\newcommand{\pp}{\mathfrak{p}} 

\newcommand{\M}{\mathbb{M}}
\newcommand{\oW}{\overline{W}} 
\newcommand{\cC}{\mathcal{C}}

\newcommand{\bM}{\overline{\M}}
\newcommand{\baR}{\overline{R}}
\newcommand{\baA}{\overline{A}}

\newcommand{\Mod}{\opn{Mod}}
\newcommand{\modd}{\opn{mod}} 
\newcommand{\stMod}{{}^*\!\Mod}
\newcommand{\stmod}{{}^*\!\modd}

\newcommand{\Hom}{\opn{Hom}}
\newcommand{\Ext}{\opn{Ext}}
\newcommand{\Spec}{\opn{Spec}}
\newcommand{\stSpec}{{}^*\!\Spec}

\newcommand{\height}{\opn{ht}}

\newcommand{\relint}{\opn{rel-int}}

\newcommand{\cpx}[1]{#1^{\bullet}} 

\newcommand{\Rone}{(\mathrm{R}_1)}

\newcommand{\set}[1]{\{#1\}}
\newcommand{\with}{\mid}

\begin{document}

\title[The canonical module and Serre's $\Rone$.]{A Canonical module characterization of Serre's~$\Rone$.}

\author{Lukas Katth\"an}

\address{Goethe-Universit\"at Frankfurt, Institut f\"ur Mathematik, 60054 Frankfurt am Main, Germany}
\email{katthaen@math.uni-frankfurt.de}

\author{Kohji Yanagawa}

\address{Department of Mathematics, Kansai University,
	Suita 564-8680, Japan} \email{yanagawa@kansai-u.ac.jp}



\thanks{The second author was supported by the JSPS KAKENHI 25400057.}

\begin{abstract}
In this short note, we give a characterization of domains satisfying Serre's condition $({\rm R}_1)$ in terms of their canonical modules.
In the special case of toric rings, this generalizes a result of the second author \cite{Y15} where the normality is described in terms of the ``shape'' of the canonical module.
\end{abstract}

\maketitle

\section{Introduction}
Let $A$ be a noetherian (commutative) domain.
Recall that $A$ is said to satisfy Serre's condition $\Rone$ if all localizations $A_\pp$ at prime ideals $\pp$ of height at most one are regular local rings.
In the present note, we characterize this condition in terms of their canonical modules under mild technical conditions (i)--(iii) on $A$ 
(see  \S2 for the detail of these conditions).
Our main result is the following:
\begin{restatable*}{thm}{Ronegeneral}
\label{R1 general}
Let $A$ be a noetherian domain satisfying (i)--(iii) and let $\baA$ denote the integral closure of $A$.
Then following are equivalent:
	\begin{enumerate}[(1)]
	\item The ring $A$ satisfies Serre's $\Rone$. 
	\item There is a canonical module of $\baA$ which is also a canonical module of $A$.
	\item Some canonical module $C$ of $A$ has an $\baA$-module structure compatible with its $A$-module structure via the inclusion $A \hookrightarrow \baA$.
	\item For some (actually, any) canonical module $C$ of $A$, the endomorphism ring $\Hom_A(C,C)$ is isomorphic to $\baA$. 
	\end{enumerate}
\end{restatable*}

This theorem has  a  multigraded version. 
Let $R$ be a $\ZZ^n$-graded domain, such that $R_0$ is a field and $R$ is a finitely generated $R_0$-algebra.
Now $R$  admits a *canonical module (i.e., canonical modules in the graded context), which is unique and denoted by $\omega_R$. 
Similarly, the integral closure $\baR$ of $R$  also admits  a *canonical module $\omega_{\baR}$. 
\begin{restatable*}{thm}{Ronegraded}\label{R1}
	The following are equivalent:
	\begin{itemize}
		\item[(1)] $R$ satisfies Serre's $\Rone$. 
		\item[(2)] $\omega_{\baR}$ is a canonical module of $R$ (in the ungraded context).
		\item[(3)] $\omega_{\baR} \cong \omega_R$ in $\stMod R$, that is, $\omega_{\baR}$ is a *canonical module of $R$. 
	\end{itemize}
\end{restatable*}

The motivation for these results stems from a related result by the second author for toric rings.
Indeed, let $\kk$ be a field and let $\M \subset \ZZ^d$ be a positive affine monoid, i.e. a finitely generated (additive) submonoid of $\ZZ^d$ without nontrivial units.
\begin{thm}[Theorem 3.1, \cite{Y15}]\label{shape}
With the above notation, consider the monoid algebra $\kk[\M] = \bigoplus_{a\in\M} \kk x^a$. 
Then the following are equivalent:
	\begin{enumerate}[(a)]
		\item $\kk[\M]$ is normal.
		\item $\kk[\M]$ is Cohen-Macaulay and the canonical module $\omega_{\kk[\M]}$ is isomorphic to the ideal $(x^a \with a \in \M \cap \relint(\RR_{\geq0}\M) )$ of $\kk[\M]$ as (graded or ungraded) $\kk[\M]$-modules.
	\end{enumerate}
Here, for $X \subset \RR^d$, $\relint(X)$ means the relative interior of $X$. 
\end{thm}

Since the ideal $\oW_R := (x^\ba \mid \ba \in \ZZ \M \cap \relint(\RR_{\geq0}\M))$ is known to be a *canonical module of the normalization $\baR=\kk[\ZZ \M \cap \RR_{\geq0}\M]$, 
\Cref{R1} yields the following equivalence. 
\begin{itemize}
\item[(i)] $R=\kk[\M]$ satisfies Serre's $\Rone$.
\item[(ii)] $\oW_R$ is a canonical module of $R$. 
\end{itemize}

The above fact is clearly analogous to Theorem~\ref{shape}, while it uses $\oW_R$ instead of $W_R := (x^\ba \mid \ba \in \M \cap \relint(\RR_{\geq0}\M))$.
We also obtain a direct generalization of \Cref{shape} as follows. 

\begin{restatable*}{thm}{thmWR}
\label{thm:W_R}
	Let $\M$ be a (not necessarily positive) affine monoid and let $C$  be a canonical module of $R = \kk[\M]$.
	Then $R$ satisfies Serre's $\Rone$ if and only if there is an injection $W_R \hookrightarrow C$ with $\dim (C / W_R) < d-1$.
	Here $d$ is the height of the *maximal ideal of $R$.  
\end{restatable*}

\section{General context}
\begin{dfn}\label{canonical module in the local case}
	Let $(A,\mm, K)$ be a noetherian local ring of dimension $d$, and $C$ a finitely generated $A$-module. 
	We say $C$ is a \emph{canonical module} of $A$, if we have an isomorphism $\Hom_A(C, E(K)) \cong H_\mm^d(A)$,
	where $E(K)$ is the injective hull of the residue field $K=A/\mm$.
	If $A$ is not local, we say that a finitely generated $A$-module $C$ is a canonical module, 
	if the localization $C_\mm$ is a canonical module of $A_\mm$ for all maximal ideals $\mm$ of $A$.
\end{dfn}

If $A$ is local, a canonical module $C$ is unique up to isomorphism (if it exists). 
In the general case, it is not necessarily unique.  
In fact, for a rank one projective module $M$, $C \otimes_A M$ is a canonical module again. 

In this section, unless otherwise specified, 
$A$ is a noetherian integral domain satisfying the following conditions:
\begin{itemize}
\item[(i)] For all maximal ideals $\mm$ of $A$, we have $\dim A_\mm = d$. 
\item[(ii)] The integral closure $\baA$ is a finitely generated $A$-module.  (A basic reference of this condition is \cite[Chapter 12]{M}.)
\item[(iii)] $A$ admits a dualizing complex $\cpx D_A$. 
(In the sequel, $\cpx D_A$ will mean the {\it normalized} dualizing complex of $A$.) 
\end{itemize}
These are mild conditions.
In fact, an integral domain which is finitely generated over a field
satisfies them.
We also remark that $H^{-d}(\cpx D_A)$ is a canonical module of $A$. 
Moreover, if $C$ is a canonical module of $A$, then the localization $C_\pp$ is a canonical module of $A_\pp$ for all prime ideals $\pp$ by 
\cite[Corollary~4.3]{A}, because $A$ is a domain.

 Somewhat surprisingly,  the following basic fact does not appear in the literature. 

\Ronegeneral

\begin{proof} 
(1) $\Rightarrow$ (2) : Assume that $A$ satisfies Serre's $\Rone$.
	The short exact sequence
	\[ 0 \rightarrow A \rightarrow \baA \rightarrow \baA/A \rightarrow 0 \]
	yields the following exact sequence:
	\[ 0 = \Ext^{-d}_A(\baA/A, \cpx D_A) \rightarrow \Ext^{-d}_A(\baA, \cpx D_A) \rightarrow \Ext^{-d}_A(A, \cpx D_A) \rightarrow \Ext^{-d+1}_A(\baA/A, \cpx D_A) = 0. \] 
	Since $A$ satisfies Serre's $\Rone$,  we have $(\baA / A)_\pp = 0$ for every prime $\pp$ of height one.
	Hence  $\dim \baA/A < d-1$, and the outer terms of the above sequence vanish. 
      To see this, for a maximal ideal $\mm$ of $A$, note that $\Ext_A^i(\baA/A, D_A^\bullet) \otimes_A A_\mm \cong \Ext_{A_\mm}^i((\baA/A)_\mm, D_{A_\mm}^\bullet)$, where $D_{A_\mm}^\bullet$ is the dualizing complex  of the local ring $A_\mm$. Hence we may assume that 
$A$ is  a local ring with the maximal ideal $\mm$. Then, the Matlis dual of  $\Ext_A^{-i}(\baA/A, D_A^\bullet)$ is the local cohomology module  $H_\mm^i(\baA/A)$. 
	
      Anyway, it follows that $\Ext^{-d}_A(\baA, \cpx D_A) \cong  \Ext^{-d}_A(A, \cpx D_A) \cong H^{-d}(\cpx D_A)$, and thus $\Ext^{-d}_A(\baA, \cpx D_A)$ is a canonical module of $A$. 
    At the same time, since $\baA$ is a finitely generated as an $A$-module, $\Ext^{-d}_A(\baA, \cpx D_A)$ is a canonical module of $\baA$. 

(2) $\Rightarrow$ (3) : Clear.  

(3) $\Rightarrow$ (4) : 
Let $C$ be a canonical module of $A$. 
Since $C$ is a torsionfree $A$-module of rank 1, it can be regarded as an $A$-submodule of  the field of fractions $Q$ of $A$, and $\Hom_A(C,C)$ can be identified with $(C: C) :=\{ \, \alpha \in Q \mid \alpha C \subseteq C \, \}$. 
Moreover, since $C$ is finitely generated, 
every $\alpha \in (C:C) \cong \Hom_A(C,C)$ is integral over $A$ (i.e., $\alpha \in \baA$) by the Cayley-Hamilton theorem, hence $(C:C) \subseteq \baA$.

By assumption, there exists a canonical module $C'$ of $A$ which admits an $\baA$-module structure.
So for $C'$ it holds that $\baA \subseteq (C':C') \subseteq \baA$.
Further, for every maximal ideal $\mm$ of $A$ it holds that $C_\mm = C'_\mm$ and hence
$(C_\mm:C_\mm) = (\baA)_\mm$. 
Here $(\baA)_\mm$ is  the localization of  $\baA$  at the multiplicatively closed set $A \setminus \mm$. 
Thus, it follows that $\Hom_A(C,C) \cong (C:C) = \baA$.

(4) $\Rightarrow$ (1) : Let $\pp$ be  a height one prime ideal of $A$.  Since $A$ is a domain, the localization $A_\pp$ is 
Cohen-Macaulay.  Hence we have 
$$\overline{A_\pp} = (\baA)_\pp \cong [\Hom_A(C,C)]_\pp \cong \Hom_{A_\pp}(C_\pp, C_\pp) \cong A_\pp.$$
This means that $A$ satisfies $\Rone$. 
\end{proof}

In the next result, we assume that $A$ is a noetherian \emph{local} ring satisfying the conditions (ii) and (iii) above. 
The localization of an integral domain which is finitely generated over a field, and a complete local domain are typical example of 
these rings.
The \emph{$({\rm S}_2)$-ification} of a local ring $A$ s the endomorphism ring 
$\Hom_A(C,C)$ of its (unique) canonical module $C$.  See \cite{AG} for detail.

\begin{cor}
Let $A$  be a local ring satisfying the conditions (ii) and (iii) above, and $A'$ its $({\rm S}_2)$-ification. 
Then $A$ satisfies Serre's $({\rm R}_1)$ if and only if so does $A'$ (equivalently, $A'$ is normal). 
\end{cor}

\begin{proof}
Sufficiency: Clear from the implication (1) $\Rightarrow$ (4) of Theorem~\ref{R1 general}. 

Necessity:  Take a height 1 prime $\pp$ of $A$. 
Let $C$ be a canonical module of $A$, then $C_\pp$ is a canonical module of $A_\pp$, and 
$A_\pp \cong \Hom_{A_\pp}(C_\pp, C_\pp) \cong [\Hom_A(C,C)]_\pp = A'_\pp$. 
If $A'$ satisfies $({\rm R}_1)$, then the localization $A'_\pp$ also. 
\end{proof}

\section{Multigraded context}
In this section, let $R=\bigoplus_{\ba \in \ZZ^n} R_\ba$ be a $\ZZ^n$-graded domain.
We assume that the degree zero part $R_0 = \kk$ is a field and that $R$ is finitely generated as a $\kk$-algebra. 
Note that the ideal $\mm$ which is generated by all homogeneous non-units is a proper ideal, which contains all proper homogeneous ideals.
In other words, $R$ is a \emph{*local ring} with \emph{*maximal ideal} $\mm$.
Note that $\mm$ is not necessarily a maximal ideal in the usual sense, though it is always a prime ideal.

The integral closure $\baR$ of $R$ is also a $\ZZ^n$-graded *local ring with $\baR_0$ a field. 
The *maximal ideal $\nn$ of $\baR$ satisfies $\nn =\sqrt{\mm \baR}$ and $\height \nn= \height \mm$.

Let  $\stMod R$ be the category of $\ZZ^n$-graded $R$-modules, and  $\stmod R$ its full subcategory consisting of finitely generated modules. For $M \in \stMod R$, let 
\[
	M^\vee :=  \Hom_R(M, {}^*\! E(R/\mm)) \cong \bigoplus_{\ba \in \ZZ^n}\Hom_\kk(M_{-\ba}, \kk)
\] 
be the $\ZZ^n$-graded Matlis dual of $M$, where $ {}^*\! E(R/\mm)$ is the injective hull of $R/\mm$ in $\stMod R$. 
For the information on this duality, see \cite[pp.312--313]{BS}. 

A module $C \in \stmod R$ is called a \emph{*canonical module} of $R$ if it satisfies $C^\vee \cong H_\mm^d(R)$ (cf.~\cite[\S 14]{BS}),
where $d := \height \mm$. 
But $R_0$ is a field, so it holds that $M \cong M^{\vee\vee}$ for all $M \in \stmod R$, see \cite[p.313]{BS}.
Hence we can take the Matlis dual of the defining equation and obtain that
\[
	\omega_R := H_\mm^d(R)^\vee
\]
is the unique *canonical module of $R$.

\Ronegraded
\begin{proof}
$R$ can be considered as a graded quotient ring of a partial Laurent polynomial ring $S=\kk[x_1^{\pm 1}, \ldots, x_l^{\pm 1}, y_1, \ldots, y_m]$ admitting a $\ZZ^n$-grading.
So by graded Local Duality \cite[Theorem~14.4.1]{BS}, there is some $\ba_0 \in \ZZ^n$ with $\Ext^c_S(R, S(\ba_0)) \cong \omega_R$, where $c=\dim S-\dim R$. Hence $\omega_R$ is also a canonical $R$-module in the ungraded context (i.e., in the sense of Definition~\ref{canonical module in the local case}).
So the implications $(1) \Leftrightarrow (2) \Leftarrow (3)$ follow directly from \Cref{R1 general}.
The implication $(1) \Rightarrow (3)$ can be proved by a similar way to the implication $(1) \Rightarrow (2)$ of \Cref{R1 general}.
For this, note that $\Ext^{c}_S(\baR, S(\ba_0)) \cong \omega_{\baR}$ as $R$-modules.
It follows from applying the graded local duality theorem to the $R$-module $\baR$.
\end{proof}

\section{Toric context}
In this section, we consider the situation that $R$ is a toric ring, i.e.,
$R = \kk[\M]  = \bigoplus_{a\in\M} \kk x^a $ for some affine monoid $\M \subseteq \ZZ^n$.
Let $\cC := \RR_{\ge 0} \M \subset \RR^n = \RR \otimes_\ZZ \ZZ^n$ be the polyhedral cone spanned by $\M$, 
and $\relint(\cC)$ its  relative interior. 
Set 
\[
	W_R := (x^\ba \mid \ba \in \M \cap \relint(\cC)),
\]
which is a $\ZZ^n$-graded ideal of $R$. 
In \cite[Theorem 3.1]{Y15}, the second author showed that,  for a Cohen-Macaulay toric ring $R$, it is normal if and only if 
$W_R$ is a canonical module. \Cref{thm:W_R} below considerably generalizes this result.

\begin{rem}
Set $\bM:=\ZZ \M \cap \cC$. 
Then $\baR = \kk[\bM]$ is Cohen-Macaulay and the ideal
\[
	\oW_R := (x^\ba \mid \ba \in \bM \cap \relint(\cC))
\]
is its *canonical module $\omega_{\baR}$.
So \Cref{R1} specializes to the statement that $R$ satisfies $\Rone$,  if and only if $\oW_R$ is a *canonical module of $R$, 
and if and only if $\oW_R$ is a canonical module of $R$ in the ungraded context. 
The first equivalence  is implicitly  stated  in \cite{SS} and the previous work \cite{K} of the first author. 
The proofs use explicit computation of the local cohomology $H_\mm^i(R)$.  
\end{rem}

\thmWR

\begin{proof} 
First, assume that $R$ satisfies $\Rone$.
The canonical inclusion $R \hookrightarrow \baR$ restricts to a homomorphism $W_R \hookrightarrow \oW_R$.
Serre's $\Rone$ implies that $\dim \baR / R < d-1$ (cf. the proof of Theorem \ref{R1 general}) 
and thus $\dim \oW_R / W_R < d -1$.
Moreover, $\oW_R$ is a canonical module of $\baR$, so by Theorem \ref{R1} it is a canonical module of $R$ as well, so the claim follows.

Next, assume that there is an inclusion $W_R \hookrightarrow C$ with $\dim C /  W_R < d -1$.
For a prime $\pp \in \Spec R$, we denote by $\pp^*$ the ideal generated by the homogeneous elements in $\pp$.
It is known that $\pp^*$ is again a prime ideal and that $R_\pp$ is regular if and only if $R_{\pp^*}$ is regular, cf. \cite[Propsition 1.2.5]{GW}.

Consider a prime ideal $\pp$ of height one.
If $\pp$ is not homogeneous, then $\pp^* = (0)$ and thus $R_{\pp^*}$ is the field of fractions of $R$. So $R_\pp$ is regular in this case.
On the other hand, if $\pp$ is homogeneous then our assumption implies that $C_\pp = (W_R)_\pp = \pp R_\pp$ is a canonical module of $R_\pp$,
For the second equality we use the fact that $$W_R = \bigcap \set{ \pp \in \stSpec R \with \height \pp > 0},$$
where $\stSpec R$ is the set of $\ZZ^n$-graded prime ideals of $R$. 

Further $R_\pp$ is a one-dimensional domain and thus Cohen-Macaulay, so the injective dimension of the canonical module $\pp R_\pp$ is finite, hence $R_\pp$ is regular by \cite[Corollary~2.3]{GHI}. It means that $R$ satisfies $\Rone$. 
\end{proof}

\begin{cor}\label{normal} Assume that $R$ satisfies Serre's $({\rm S}_2)$. Then the following are equivalent. 
\begin{itemize}
\item[(1)] $R$ is normal. 
\item[(2)] $W_R$ is a canonical module of $R$ (in the ungraded context).  
\item[(3)] $\oW_R$ is a canonical module of $R$ (in the ungraded context).  
\end{itemize}
\end{cor}

\begin{exmp}
We give two examples to show that Theorem \ref{thm:W_R} and Corollary \ref{normal} cannot be extended.
\begin{asparaenum}
\item Even if $R$ satisfies $\Rone$,  $W_R$ may not be canonical. 
Indeed, consider the affine monoid
\[\M := \set{(a,b) \in \NN^2 \with a+b \equiv 0 \mod{2}} \setminus\set{(1,1)}. \]
It is not difficult to see that $W_R$ has three minimal generators in the degrees $(1,3),(2,2)$ and $(3,1)$.
On the other hand, the *canonical module $C$ of $R = \kk[\M]$ has only two generators, in the degrees $(1,1)$ and $(2,2)$. Thus $W_R$ is not a canonical module, even in the ungraded context.

\item Moreover, even if $W_R$ is canonical, $R$ may not be normal.  
For example, consider 
\[\M := \NN^2 \setminus \set{(1,0)}. \]
and $R = \kk[\M]$.
Here, $\M$ and thus $R$ are clearly not normal, but nevertheless $W_R$ is a canonical module of $R$.
\end{asparaenum}
\end{exmp}

\section*{Acknowledgments}
We wish to thank Professors W. Bruns, W. Vasconcelos, M. Hashimoto, P. Schenzel  and R. Takahashi 
for several helpful comments and conversations.


\begin{thebibliography}{99}
\bibitem{A} Y. Aoyama, Some basic results on canonical modules, J. Math. Kyoto Univ. {\bf 23} (1983), 85--94. 

\bibitem{AG} Y. Aoyama and S. Goto, On the endomorphism ring of the canonical module, 
J. Math. Kyoto Univ. {\bf  25}  (1985),   21--30. 




\bibitem{BS}
M.~P. Brodmann and R.~Y. Sharp.
\newblock {\em Local cohomology: an algebraic introduction with geometric
  applications, 2nd Ed.}
\newblock Cambridge university press, 2012.


\bibitem{GHI} S. Goto, F. Hayasaka, S. Iai, 
The $a$-invariant and Gorensteinness of graded rings associated to filtrations of ideals in regular local rings, 
Proc. Amer. Math. Soc. {\bf 131} (2003), 87--94

\bibitem{GW}
S. Goto and K. Watanabe, On graded rings, II (${\bf Z}^n$-graded rings), Tokyo J. Math. {\bf 1} (1978), 237--261. 

\bibitem{K}
L. Katth\"an,  Non-normal affine monoid algebras, Manuscripta Math. {\bf 146}  (2015), 223--233.  

\bibitem{M} 
H. Matsumura, {\em Commutative algebra}, W.A. Benjamin, New York (1970).

\bibitem{SS}
U. Sch\"afer and  P. Schenzel,  Dualizing complexes of affine semigroup rings. Trans. Am. Math. Soc. {\bf 322},  (1990) 561--582.  

\bibitem{Y15}
K. Yanagawa, 
Dualizing complexes of seminormal affine semigroup rings and toric face rings, 
J. Algebra {\bf 425} (2015), 367--391.

\end{thebibliography}
\end{document}